\documentclass{amsart}

\usepackage{amssymb,amsmath,amsthm,latexsym,booktabs,lscape}

\theoremstyle{definition}
\newtheorem{definition}{Definition}[section]

\newtheorem{remark}[definition]{Remark}

\theoremstyle{plain}
\newtheorem{lemma}[definition]{Lemma}
\newtheorem{proposition}[definition]{Proposition}
\newtheorem{theorem}[definition]{Theorem}
\newtheorem{corollary}[definition]{Corollary}
\newtheorem{conjecture}[definition]{Conjecture}
\newtheorem{problem}{Open Problem}

\newcommand{\s}{\mathfrak{sl}(2,\mathbb{C})}

\begin{document}

\title{The simple non-Lie Malcev algebra as a Lie-Yamaguti algebra}

\author{Murray R. Bremner}

\address{Department of Mathematics and Statistics, University of Saskatchewan, Canada}

\email{bremner@math.usask.ca}

\author{Andrew Douglas}

\address{Department of Mathematics, New York City College of Technology,
City University of New York, USA}

\email{adouglas@citytech.cuny.edu}

\begin{abstract}
The simple 7-dimensional Malcev algebra $M$ is isomorphic to the irreducible $\s$-module $V(6)$ with binary product 
$[x,y] = \alpha( x \wedge y)$ defined by the $\s$-module morphism $\alpha\colon \Lambda^2 V(6) \to V(6)$.
Combining this with the ternary product $(x,y,z) = \beta(x \wedge y) \cdot z$ defined by the $\s$-module morphism 
$\beta\colon \Lambda^2 V(6) \to V(2) \approx \s$ gives $M$ the structure of a generalized Lie triple system, or
Lie-Yamaguti algebra.  We use computer algebra to determine the polynomial identities of low degree satisfied 
by this binary-ternary structure.
\end{abstract}  

\maketitle


\section{Introduction}

Moufang-Lie algebras were introduced by Malcev \cite{Malcev} in 1955 as the tangent algebras 
of analytic Moufang loops.
These structures were given the name Malcev algebras by Sagle \cite{Sagle} in 1961; the basic 
theory was completed by Kuzmin \cite{Kuzmin} in 1968.
Malcev algebras are related to alternative algebras in the same way that Lie algebras are 
related to associative algebras: any subspace of an associative (resp.~alternative) 
algebra closed under the commutator $[x,y] = xy - yx$ is a Lie (resp.~Malcev) 
algebra.
An important recent development is the construction by P\'erez-Izquierdo 
and Shestakov \cite{PerezShestakov} of universal nonassociative enveloping algebras for Malcev 
algebras which share many of the properties of universal associative enveloping algebras 
of Lie algebras; for a survey, see Bremner et al.~\cite{BremnerSurvey}.

Lie triple systems were introduced by Jacobson \cite{Jacobson} in 1949; the basic theory was 
completed by Lister \cite{Lister} in 1952.
A simultaneous generalization of Lie algebras and Lie triple systems was introduced by Yamaguti 
\cite{Yamaguti} in 1958.  
These binary-ternary structures were called Lie-Yamaguti algebras by Kinyon and Weinstein \cite{KinyonWeinstein}
in 2001; the basic theory has recently been completed by Benito et al.~\cite{Benito1,Benito2}.

In this paper we consider polynomial identities for the 7-dimensional simple Malcev algebra
regarded as a Lie-Yamaguti algebra: we consider both the usual binary product and an
unusual ternary product.  To define this structure we use the representation
theory of the simple Lie algebra $\s$.  We start with the irreducible 7-dimensional $\s$-module $V(6)$
of highest weight 6, 
and consider the decomposition of its exterior square into irreducible summands:
  \[
  \Lambda^2 V(6) \approx V(10) \oplus V(6) \oplus V(2).
  \]
Up to nonzero scalar multiples, there are unique $\s$-module morphisms: 
  \[ 
  \alpha\colon \Lambda^2 V(6) \to V(6),
  \qquad
  \beta\colon \Lambda^2 V(6) \to V(2).
  \] 
If we define a bilinear operation on $V(6)$ by the equation $[x,y] = \alpha( x \wedge y)$, 
then we obtain a structure isomorphic to the 7-dimensional simple Malcev algebra.
If we also define a trilinear operation on $V(6)$ by the equation $(x,y,z) = \beta( x \wedge y ) \cdot z$,
where the dot denotes the action of $\s \approx V(2)$,
then combining these operations gives $V(6)$ the structure of a Lie-Yamaguti algebra.

Our goal in this paper is to use computer algebra to determine the polynomial identities of low degree
satisfied by the operations $[x,y]$ and $(x,y,z)$, independently and together.
Our main results are as follows:
  \begin{enumerate}
  \item
For the bilinear operation $[x,y]$, apart from the obvious anticommutativity, we recover a multilinear 
form of the Malcev identity in degree 4, an identity equivalent to Filippov's $h$-polynomial in degree 5, 
and we verify that every identity of degree $\le 7$ is a consequence of the identities of degree $\le 5$.
  \item
For the trilinear operation $(x,y,z)$, apart from the obvious skew-symmetry in the first two arguments,
we recover the ternary derivation identity in degree 5, and discover a 357-dimensional space of identities
of degree 7 which are not consequences of the identities of lower degree; we present an
explicit 141-term identity of degree 7.
  \item
For both operations together, in addition to the identities already mentioned, we recover the defining
identities in degrees 3 and 4 for Lie-Yamaguti algebras, together with a new 31-term identity of degree 5.
  \end{enumerate}
Section \ref{preliminaries} recalls the definitions of Malcev and Lie-Yamaguti algebras, summarizes the
necessary background on representations of $\s$, and presents explicit constructions 
of the binary and ternary products on $V(6)$.
Section \ref{computational} recalls some basic concepts in the combinatorics of free nonassociative algebras,
and describes our computational methods for studying polynomial identities.
Sections \ref{sectionbinary}, \ref{sectionternary}, and \ref{sectionmixed} present our original results for 
the bilinear operation $[x,y]$, the trilinear operation $(x,y,z)$, 
and the two operations together.


\section{Preliminaries} \label{preliminaries}

\subsection{Malcev algebras and Lie-Yamaguti algebras}

We recall the definitions of our main objects of study.
We assume the base field has characteristic $\ne 2$.

\begin{definition} \label{definitionMalcev}
(Malcev \cite{Malcev}, Sagle \cite{Sagle}, Kuzmin \cite{Kuzmin})
A \textbf{Malcev algebra} is a vector space with a bilinear product $[a,b]$ 
satisfying these multilinear identities:
  \allowdisplaybreaks
  \begin{align*}
  &
  [a,b] + [b,a] 
  \equiv 
  0,
  \\
  &
  [[[a,c],b],d]+[[[d,a],c],b]+[[[c,b],d],a]+[[[b,d],a],c]-[[a,b],[c,d]]
  \equiv
  0.
  \end{align*}
\end{definition}

Every Lie algebra is a Malcev algebra.  
Over an algebraically closed field of characteristic 0, there is a unique simple non-Lie Malcev
algebra, isomorphic to the commutator algebra of the 7-dimensional space of octonions
with zero real part.

\begin{definition} \label{definitionLY}
(Yamaguti \cite{Yamaguti}, Kinyon and Weinstein \cite{KinyonWeinstein}, Benito et al.~\cite{Benito1,Benito2})
A \textbf{Lie-Yamaguti algebra} is a vector space with bilinear and
trilinear products $[a,b]$ and $(a,b,c)$ satisfying these multilinear identities:
  \allowdisplaybreaks
  \begin{align*}
  &
  [a,b] + [b,a]
  \equiv 0,
  \\
  &
  (a,b,c) + (b,a,c)
  \equiv 0,
  \\
  &
  [[a,b],c] + [[b,c],a] + [[c,a],b] + (a,b,c) + (b,c,a) + (c,a,b)
  \equiv 0,
  \\
  &
  ( [a,b], c, d ) + ( [b,c], a, d ) + ( [c,a], b, d )
  \equiv 0,
  \\
  &
  [ (a,b,c), d ] - [ (a,b,d), c ] - (a,b,[c,d])
  \equiv 0,
  \\
  &
  ((a,b,c),d,e) - ((a,b,d),c,e) - (a,b,(c,d,e)) + (c,d,(a,b,e))
  \equiv 0.
  \end{align*}
\end{definition}

Every Lie algebra is a Lie-Yamaguti algebra if we define $(a,b,c) \equiv 0$.
Every Lie triple system is a Lie-Yamaguti algebra if we define $[a,b] \equiv 0$.
These structures arise naturally as follows; see Kinyon and Weinstein \cite{KinyonWeinstein}:   
Let $\mathfrak{g}$ be a reductive Lie algebra with decomposition 
$\mathfrak{g} = \mathfrak{h}\oplus \mathfrak{m}$ such that 
$[ \mathfrak{h}, \mathfrak{m} ] \subseteq \mathfrak{m}$ and  
$[ \mathfrak{h}, \mathfrak{h} ] \subseteq \mathfrak{h}$.   
Let $\pi_{\mathfrak{h}}\colon \mathfrak{g} \longrightarrow \mathfrak{h}$ and 
$\pi_{\mathfrak{m}}\colon \mathfrak{g} \longrightarrow \mathfrak{m}$ 
be the projections onto $\mathfrak{h}$ and $\mathfrak{m}$.  
The subspace $\mathfrak{m}$ is a Lie-Yamaguti algebra with respect to the following operations:
  \[ 
  [a,b] = \pi_{\mathfrak{m}}( [a,b] ),
  \qquad
  (a,b,c) = [ \, \pi_{\mathfrak{h}}( [a,b] ), \, c \, ].
  \]

\subsection{Irreducible representations of $\s$}

The 3-dimensional simple Lie algebra $\s$ consists of the $2 \times 2$ complex matrices of trace zero with 
respect to the commutator $[A,B] = AB - BA$.  The standard basis of $\s$ and the commutation relations among basis
elements are as follows:
  \begin{align*}
  &
  H =
  \left[ \begin{array}{rr}
  1 &\!\!\!  0 \\
  0 &\!\!\! -1 
  \end{array} \right], 
  \quad
  E = 
  \left[ \begin{array}{rr}
  0 & 1 \\
  0 & 0
  \end{array} \right], 
  \quad
  F = 
  \left[ \begin{array}{rr}
  0 & 0 \\
  1 & 0
  \end{array} \right], 
  \quad 
  \begin{array}{l}
  {[H,E] = 2E}, \\
  {[H,F] = -2F}, \\
  {[E,F] = H}. 
  \end{array} 
  \end{align*}
We recall the classification of irreducible finite-dimensional $\s$-modules; 
see Erdmann and Wildon \cite{ErdmannWildon} or Humphreys \cite{Humphreys}.  
Any such module $V$ decomposes into weight spaces according to the action of $H$:
  \[
  V = \bigoplus_{\lambda \in \mathbb{C}} V_\lambda, 
  \qquad
  V_{\lambda} = \{ \, v \in V \mid H \cdot v = \lambda v \, \}.
  \]
The commutation relations for $\s$ imply
$F \cdot V_\lambda \subseteq V_{\lambda-2}$
and
$E \cdot V_\lambda  \subseteq V_{\lambda+2}$.
Since $V$ is finite-dimensional, there exists $V_\lambda \ne \{0\}$ such that $V_{\lambda+2} = \{0\}$.  
Choose $v_0 \in V_\lambda$, $v_0 \ne 0$; then $E \cdot v_0 = 0$.     
Define weight vectors ($H$-eigenvectors) as follows:
  \begin{equation}
  \label{V(n)basis}
  v_i = \frac{1}{i!} F^i \cdot v_0 \quad (i \ge 0).
  \end{equation}
For $i \ge 0$ we have
  \begin{equation}
  \label{V(n)}
  H \cdot v_{i} = (\lambda-2i) \, v_{i}, 
  \quad
  E \cdot v_{i} = (\lambda-i+1) \, v_{i-1},
  \quad
  F \cdot v_{i} = (i+1) \, v_{i+1}.
  \end{equation}
The action of $H$ implies that the nonzero $v_i$ are linearly independent.   
Since $V$ is finite-dimensional, there exists $n \ge 0$ such that $v_n \ne 0$ and $v_m = 0$ for $m > n$.
Setting $i = n+1$ in the action of $E$ gives $0 = (\lambda-n) v_n$; hence $\lambda = n \in \mathbb{Z}$. 
The submodule with basis $\{ v_0, v_1, \dots, v_n \}$ must equal $V$ since $V$ is irreducible.

\begin{theorem}
For every $n \ge 0$, there exists an irreducible $\s$-module $V(n)$ of dimension $n+1$, 
generated by a highest weight vector $v_0$ with $H \cdot v_0 = n v_0$, and having a basis $\{ v_0, v_1, \dots, v_n \}$ 
satisfying equations \eqref{V(n)} with $\lambda = n$.
Conversely, any irreducible finite-dimensional $\s$-module is isomorphic to $V(n)$ for some $n$.
\end{theorem}

It will be convenient to index the basis vectors for $V(n)$ by their $H$-eigenvalues; we obtain
the basis $\{ v_n, v_{n-2}, \dots, v_{-n+2}, v_{-n} \}$ with the following action of $H, E, F$:
  \begin{align}
  H \cdot v_{n-2i} &= (n-2i) \, v_{n-2i}, 
  \notag
  \\
  E \cdot v_{n-2i} &= (n-i+1) \, v_{n-2(i-1)},
  \label{V(n)weight}
  \\ 
  F \cdot v_{n-2i} &= (i+1) \, v_{n-2(i+1)}.
  \notag
  \end{align}
For $n = 2$ we have the adjoint module $V(2) \approx \s$.  However, we must be careful because 
equations \eqref{V(n)weight} give the following action of $\s$ on $V(2)$:
  \[
  \begin{array}{lll}
  H \cdot v_2 = 2 v_2, 
  &\qquad
  H \cdot v_0 = 0,
  &\qquad
  H \cdot v_{-2} = -2 v_{-2},
  \\
  E \cdot v_2 = 0, 
  &\qquad
  E \cdot v_0 = 2 v_2,
  &\qquad
  E \cdot v_{-2} = v_0,
  \\  
  F \cdot v_2 = v_0, 
  &\qquad
  F \cdot v_0 = 2 v_{-2},
  &\qquad
  F \cdot v_{-2} = 0.
  \end{array}
  \]
If we replace $v_2$ by $-v_2$ then we obtain
  \[
  \begin{array}{lll}
  H \cdot (-v_2) = 2 (-v_2), 
  &\qquad
  H \cdot v_0 = 0,
  &\qquad
  H \cdot v_{-2} = -2 v_{-2},
  \\
  E \cdot (-v_2) = 0, 
  &\qquad
  E \cdot v_0 = -2 (-v_2),
  &\qquad
  E \cdot v_{-2} = v_0,
  \\  
  F \cdot (-v_2) = - v_0, 
  &\qquad
  F \cdot v_0 = 2 v_{-2},
  &\qquad
  F \cdot v_{-2} = 0.
  \end{array}
  \] 
Thus the isomorphism $\s \approx V(2)$ includes a sign change for $E$:
  \begin{equation}
  \label{signchange}
  E \leftrightarrow -v_2,
  \qquad
  H \leftrightarrow v_0,
  \qquad
  F \leftrightarrow v_{-2}.    
  \end{equation}
 
\subsection{Tensor products of modules}

Every finite-dimensional $\s$-module is completely reducible.
A classical result gives the decomposition of the tensor product of two irreducible modules.

\begin{theorem} 
\emph{(Clebsch-Gordan)}
If  $m \ge n \ge 0$ then
  \[
  V(m) \otimes V(n) \, \approx \, \bigoplus_{i=0}^{n} V(m{+}n{-}2i).
  \]
\end{theorem}

If $m = n$ then the tensor product is the direct sum of symmetric and exterior squares:
$V(n) \otimes V(n) \approx  S^2 V(n) \oplus \Lambda^2 V(n)$.

\begin{corollary} \label{CGcorollary}
If $n \ge 0$ then
  \[
  S^{2} V(n) \, \approx \bigoplus_{k=0}^{\lfloor n/2 \rfloor} V(2n{-}4k),
  \qquad
  \Lambda^{2} V(n) \, \approx \bigoplus_{k=0}^{\lfloor (n-1)/2 \rfloor} V(2n{-}2{-}4k).
  \]
\end{corollary}

In this paper we are concerned primarily with $V(6)$; we have
  \[
  S^2 V(6) \approx V(12) \oplus V(8) \oplus V(4) \oplus V(0),
  \qquad
  \Lambda^2 V(6) \approx V(10) \oplus V(6) \oplus V(2).
  \]
In particular, both $V(6)$ and the adjoint module $V(2)$ occur as summands of the exterior square,
providing the following $\s$-module morphisms, which are unique up to nonzero scalar multiples:
  \[
  \alpha\colon \Lambda^2 V(6) \to V(6),
  \qquad
  \beta\colon \Lambda^2 V(6) \to V(2).
  \]

\begin{theorem}
\emph{(Bremner and Hentzel \cite{BremnerHentzel})}
The bilinear product $[x,y] = \alpha( x \wedge y )$ makes $V(6)$ 
isomorphic to the 7-dimensional simple non-Lie Malcev algebra.
\end{theorem}

We consider this bilinear product and the trilinear product 
$(x,y,z) = \beta( x \wedge y ) \cdot z$ induced by the composition of the morphism $\beta$ 
with the natural action of the adjoint module $V(2) \approx \s$ on $V(6)$.
We use the representation theory of $\s$ to calculate the structure constants; 
a more traditional approach using classical invariant theory has been described by Dixmier \cite{Dixmier}.

\subsection{Bilinear and trilinear products on $V(6)$}

In order to calculate explicitly the products on $V(6)$,
we first determine a generating vector in each summand of the exterior square.
We then use the action of $F$ to determine a basis of $H$-eigenvectors for each irreducible summand. From this we form 
the transition matrix from the ``module basis'' of $V(10) \oplus V(6) \oplus V(2)$ to the ``tensor basis'' of
$\Lambda^2 V(6)$.
Inverting this matrix gives the transition matrix from the tensor basis to the module basis, and from this 
we obtain explicit projections from $\Lambda^2 V(6)$ onto $V(6)$ and $V(2)$. 

Consider the standard basis $\{ v_6, v_4, v_2, v_0, v_{-2}, v_{-4}, v_{-6} \}$ of $H$-eigenvectors for 
the irreducible $\s$-module $V(6)$.  
The action of $H, E, F$ is as follows:
  \begin{equation}
  \label{HEFactionV(6)}
  \begin{array}{r|rrrrrrr}
  & 
  v_6 & v_4 & v_2 & v_0 & v_{-2} & v_{-4} & v_{-6}
  \\
  \midrule
  H &
  6 v_6 & 4 v_4 & 2 v_2 & 0 & -2 v_{-2} & -4 v_{-4} & -6 \,v_{-6}
  \\
  E &
  0 & 6 v_6 & 5 v_4 & 4 v_2 & 3 v_0 & 2 v_{-2} & v_{-4}
  \\
  F & 
  v_4 & 2 v_2 & 3 v_0 & 4 v_{-2} & 5 v_{-4} & 6 v_{-6} & 0
  \end{array}
  \end{equation}
  
\begin{definition}
The \textbf{tensor basis} of $\Lambda^2 V(6)$ consists of the 21 elements
  \[
  v_i \wedge v_j
  \quad
  ( \,
  i > j; \,
  i,j \in \{6,4,2,0,-2,-4,-6\} \,
  ),
  \]
where
$v_i \wedge v_j$ precedes $v_{i'} \wedge v_{j'}$ if either (1) $i > i'$ or (2) $i = i'$ and $j > j'$.
\end{definition}

The summand $V(10)$ is generated by $s_{10} = v_6 \wedge v_4$; the other basis vectors 
are found by applying $F$ using equation \eqref{V(n)basis} and the derivation rule, 
  \[
  F \cdot ( x \wedge y ) = ( F \cdot x ) \wedge y + x \wedge ( F \cdot y ).
  \]
A generator $t_6$ for the summand $V(6)$ must be a linear combination of $v_6 \wedge v_0$ and $v_2 \wedge v_4$;
imposing the condition $E \cdot t_6 = 0$ and proceeding as before we obtain the other basis vectors.
A similar calculation produces a basis for the summand $V(2)$.
All these basis vectors are displayed in Table \ref{modulebasistable}.

\begin{definition}
The \textbf{module basis} for $\Lambda^2 V(6)$ consists of the preceding vectors in the order
in which they were computed:
  \[
  s_{10}, \, s_8, \, \dots, s_{-8}, \, s_{-10}, \,
  t_6, \, t_4, \, \dots, t_{-4}, \, t_{-8}, \,
  u_2, \, u_0, \, u_{-2}.
  \]
\end{definition}

Let $A$ be the $21 \times 21$ matrix in which the $(i,j)$ entry is the coefficient of 
the $i$-th tensor basis vector 
in the formula for the $j$-th module basis vector.
The columns of the inverse matrix express the tensor basis vectors as linear combinations of 
the module basis vectors.
In particular, rows 12--18, respectively 19--21, determine the projections $\alpha\colon \Lambda^2 V(6) \to V(6)$,
respectively $\beta\colon \Lambda^2 V(6) \to V(2)$.
From this we obtain Tables \ref{alphatable} and \ref{betatable} for the $\s$-module morphisms $\alpha$ and $\beta$.
The entry $c_{pq}$ in row $v_p$ and column $v_q$ indicates respectively that
  \[ 
  \alpha( v_p \wedge v_q ) = c_{pq} t_{p+q},
  \qquad 
  \beta( v_p \wedge v_q ) = c_{pq} u_{p+q}.
  \] 
We have scaled the entries so that they are relatively prime integers.

  \begin{table}
  \begin{align*}
  s_{10} &= v_6 \wedge v_4, 
  \\
  s_8 &= 2 \, v_6 \wedge v_2, 
  \\
  s_6 &= 3 \, v_6 \wedge v_0 + v_4 \wedge v_2,
  \\
  s_4 &= 4 \, v_6 \wedge v_{-2} + 2 \, v_4 \wedge v_0,
  \\
  s_2 &= 5 \, v_6 \wedge v_{-4} + 3 \, v_4 \wedge v_{-2} + v_2 \wedge v_0,
  \\
  s_0 &= 6 \, v_6 \wedge v_{-6} + 4 \, v_4 \wedge v_{-4} + 2 \, v_2 \wedge v_{-2},
  \\
  s_{-2} &= 5 \, v_4 \wedge v_{-6} + 3 \, v_2 \wedge v_{-4} + v_0 \wedge v_{-2},
  \\
  s_{-4} &= 4 \, v_2 \wedge v_{-6} + 2 \, v_0 \wedge v_{-4},
  \\
  s_{-6} &= 3 \, v_0 \wedge v_{-6} + v_{-2} \wedge v_{-4},
  \\
  s_{-8} &= 2 \, v_{-2} \wedge v_{-6},
  \\
  s_{-10} &= v_{-4} \wedge v_{-6};
  \\
  &
  \\
  t_6 &= 3 \, v_6 \wedge v_0 - 2 \, v_4 \wedge v_2, 
  \\   
  t_4 &= 12 \, v_6 \wedge v_{-2} - 3 \, v_4 \wedge v_0,    
  \\
  t_2 &= 30 \, v_6 \wedge v_{-4} - 3 \, v_2 \wedge v_0,    
  \\
  t_0 &= 60 \, v_6 \wedge v_{-6} + 10 \, v_4 \wedge v_{-4} - 4 v_2 \wedge v_{-2},    
  \\
  t_{-2} &= 30 \, v_4 \wedge v_{-6} - 3 \, v_0 \wedge v_{-2},    
  \\
  t_{-4} &= 12 \, v_2 \wedge v_{-6} - 3 \, v_0 \wedge v_{-4},    
  \\
  t_{-6} &= 3 \, v_0 \wedge v_{-6} - 2 v_{-2} \wedge v_{-4};
  \\ 
  &
  \\  
  u_2 &= 15 \, v_6 \wedge v_{-4} - 5 \, v_4 \wedge v_{-2} + 3 \, v_2\wedge v_0,    
  \\
  u_0 &= 90 \, v_6 \wedge v_{-6} - 10 \, v_4 \wedge v_{-4} + 2 \, v_2 \wedge v_{-2},
  \\    
  u_{-2} &= 15 \, v_4 \wedge v_{-6} - 5 \, v_2 \wedge v_{-4} + 3 \, v_0 \wedge v_{-2}.   
  \end{align*}
  \caption{The module basis for $\Lambda^2 V(6)$}
  \label{modulebasistable}
  \end{table}

  \begin{table}
  \[
  \begin{array}{r|rrrrrrr} 
  \alpha & v_6 & v_4 & v_2 & v_0 & v_{-2} & v_{-4} & v_{-6} 
  \\ 
  \midrule
  v_6    &   . &  . &   . &  20 &  10 &   4 &  1 
  \\ 
  v_4    &   . &  . & -60 & -20 &   . &   6 &  4 
  \\ 
  v_2    &   . & 60 &   . & -20 & -15 &   . & 10 
  \\ 
  v_0    & -20 & 20 &  20 &   . & -20 & -20 & 20 
  \\ 
  v_{-2} & -10 &  . &  15 &  20 &   . & -60 & . 
  \\ 
  v_{-4} &  -4 & -6 &   . &  20 &  60 &   . & . 
  \\ 
  v_{-6} &  -1 & -4 & -10 & -20 &   . &   . & .
  \end{array}
  \]
  \caption{The $\s$-module morphism $\alpha\colon \Lambda^2 V(6) \to V(6)$}
  \label{alphatable}  
  \[
  \begin{array}{r|rrrrrrr} 
  \beta & v_6 & v_4 & v_2 & v_0 & v_{-2} & v_{-4} & v_{-6} 
  \\ 
  \midrule
  v_6    &  . &  . &   . &   . &   . &   2 & 1 
  \\ 
  v_4    &  . &  . &   . &   . & -10 &  -4 & 2 
  \\ 
  v_2    &  . &  . &   . &  20 &   5 & -10 & . 
  \\ 
  v _0   &  . &  . & -20 &   . &  20 &   . & . 
  \\ 
  v_{-2} &  . & 10 &  -5 & -20 &   . &   . & . 
  \\ 
  v_{-4} & -2 &  4 &  10 &   . &   . &   . & . 
  \\ 
  v_{-6} & -1 & -2 &   . &   . &   . &   . & .
  \end{array}
  \]
  \caption{The $\s$-module morphism $\beta\colon \Lambda^2 V(6) \to V(2)$}
  \label{betatable}
  \end{table}

\subsection{Another construction of the morphisms $\alpha$ and $\beta$}

If $V$ is a vector space of finite dimension $N$ over a field $F$, then the endomorphism algebra
$\mathrm{End}_F(V)$ is naturally isomorphic to $V^\ast \otimes_F V$.  If we choose a basis for $V$ then
we can identity $\mathrm{End}_F(V)$ with the algebra $M_N(F)$ of $N \times N$ matrices over $F$.  We consider
the special case $V = V(n)$, the irreducible $\s$-module of dimension $N = n+1$.  We identify $V(n)$ with
the space of homogeneous polynomials of degree $n$ in $x$ and $y$; the action of $\s$ is given by these formulas:
  \[
  H = x \frac{\partial}{\partial x} - y \frac{\partial}{\partial y},
  \qquad
  E = x \frac{\partial}{\partial y},
  \qquad
  F = y \frac{\partial}{\partial x}.
  \]
We identify the basis vectors $v_{n-2i}$ of $V(n)$ with monomials as follows:
  \[
  v_{n-2i} = \binom{n}{i} x^{n-i} y^i.
  \]
This action of $\s$ satisfies equations \eqref{V(n)weight}.
Since $V(n)$ is isomorphic to its dual, we obtain an embedding of $\s$ into the general linear Lie
algebra $\mathfrak{gl}(n{+}1,\mathbb{C})$, which is an $\s$-module if we define the action by 
$A \cdot B = [A,B]$ for $A \in \s$ and $B \in \mathfrak{gl}(n{+}1,\mathbb{C})$.  Corollary \ref{CGcorollary} 
gives the decomposition of this $\s$-module into a direct sum of irreducible modules,
and the divided powers $E^i\!/i!$ ($0 \le i \le n$) of the matrix representing $E$ are 
the highest weight vectors of the summands.

In the special case $n = 4$ we obtain this embedding of $\s$ into $\mathfrak{gl}(5,\mathbb{C})$:
  \[
  H, E, F
  \mapsto
  \left[
  \begin{array}{rrrrr}
  4 & . &\, . &\!\!\!\!  . &\!\!\!\!  . \\
  . & 2 &\, . &\!\!\!\!  . &\!\!\!\!  . \\
  . & . &\, . &\!\!\!\!  . &\!\!\!\!  . \\
  . & . &\, . &\!\!\!\! -2 &\!\!\!\!  . \\
  . & . &\, . &\!\!\!\!  . &\!\!\!\! -4 
  \end{array}
  \right],
  \left[
  \begin{array}{rrrrr}
  . & 1 & . & . & . \\
  . & . & 2 & . & . \\
  . & . & . & 3 & . \\
  . & . & . & . & 4 \\
  . & . & . & . & . 
  \end{array}
  \right],
  \left[
  \begin{array}{rrrrr}
  . & . & . & . & . \\
  4 & . & . & . & . \\
  . & 3 & . & . & . \\
  . & . & 2 & . & . \\
  . & . & . & 1 & . 
  \end{array}
  \right].
  \]
The decomposition of $\mathfrak{gl}(5,\mathbb{C})$ as an $\s$-module is given abstractly as follows,
  \[
  V(4) \otimes V(4) \approx V(0) \oplus V(2) \oplus V(4) \oplus V(6) \oplus V(8),
  \] 
where the symmetric and exterior squares are
  \[
  S^2 V(4) \approx V(0) \oplus V(4) \oplus V(8),
  \qquad
  \Lambda^2 V(4) \approx V(2) \oplus V(6).
  \]
A basis for $V(0)$ is the identity matrix.
A basis for $V(2)$ consists of the images of $H, E, F$ given above.
Highest weight vectors for $V(4)$, $V(6)$, $V(8)$ are the matrices
  \[
  \frac{E^2}{2}
  =
  \left[
  \begin{array}{rrrrr}
  . & . & 1 & . & . \\
  . & . & . & 3 & . \\
  . & . & . & . & 6 \\
  . & . & . & . & . \\
  . & . & . & . & . 
  \end{array}
  \right],
  \;
  \frac{E^3}{6}
  =
  \left[
  \begin{array}{rrrrr}
  . & . & . & 1 & . \\
  . & . & . & . & 4 \\
  . & . & . & . & . \\
  . & . & . & . & . \\
  . & . & . & . & . 
  \end{array}
  \right],
  \;
  \frac{E^4}{24}  
  =
  \left[
  \begin{array}{rrrrr}
  . & . & . & . & 1 \\
  . & . & . & . & . \\
  . & . & . & . & . \\
  . & . & . & . & . \\
  . & . & . & . & . 
  \end{array}
  \right].  
  \]
The other basis matrices for these summands are obtained using equation \eqref{V(n)basis}. 
Given two matrices $A, B$ in the subspace corresponding to $V(6)$, we calculate the commutator $[A,B]$
and its projections onto the summands $V(6)$ and $V(2)$: this is another way to compute the morphisms
$\alpha$ and $\beta$.  The direct sum $V(2) \oplus V(6)$ is a subspace of
$\mathfrak{sl}(5,\mathbb{C})$ isomorphic to the orthogonal simple Lie algebra $\mathfrak{o}(5,\mathbb{C})$.


\section{Nonassociative polynomials and computational methods} \label{computational}

We recall some definitions about the combinatorics of free nonassociative structures,
and describe our computational methods for studying polynomial identities.

\begin{definition}
A \textbf{binary association type} in degree $n$ is a placement of $n{-}1$ pairs of 
anticommutative brackets $[-,-]$ in the word $x = x_1 x_2 \cdots x_n$. Each
type has a unique factorization $x = [y,z]$ where $n > \deg(y) \ge \deg(z) \ge 1$. 
Types have a \textbf{standard order}: 
$[y,z]$ precedes $[y',z']$ if 
($i$) $\deg(y) > \deg(y')$, or 
($ii$) $\deg(y) = \deg(y')$ and $y$ precedes $y'$, or 
($iii$) $y = y'$ and $\overline{z}$ precedes $\overline{z}'$ where the bar replaces $x_{p+1} \cdots x_{p+q}$ by $x_1
\cdots x_q$ for $p = \deg(y)$ and $q = \deg(z)$.
\end{definition}

\begin{definition} \label{definitionbinary1}
A \textbf{binary monomial} in degree $n$ is a permutation $\pi \in S_n$ applied to 
an association type: $\pi ( x_1 \cdots x_n ) = x_{\pi(1)} \cdots x_{\pi(n)}$,
brackets omitted. If $x$ and $x'$ are two binary monomials with the same degree and association type,
then $x$ and $x'$ are \textbf{equivalent} if anticommutativity implies $x = \pm x'$. A binary
monomial is in \textbf{normal form} if every submonomial $[y,y']$, where $y$ and $y'$ 
have the same association type and first symbols $x_i$ and $x_j$, satisfies $i < j$.
\end{definition}

\begin{definition} \label{definitionbinary3}
The space $P_n$ of \textbf{binary polynomials} in degree $n$ is the vector space 
with basis consisting of the binary monomials in normal form. This space is an $S_n$-module: 
we apply a permutation $\pi$ to the subscripts in $x$, 
preserving the association type; we obtain the monomial $x^\pi$, which in general is
not in normal form; we apply anticommutativity to obtain $\pi \cdot x = \pm x'$, where $x'$ is the
unique monomial in normal form equivalent to $x^\pi$.
\end{definition}

\begin{definition}
Let $I(x_1,x_2,\hdots,x_n)$ be a binary polynomial of degree $n$. From $I$ we obtain $n+1$
\textbf{liftings} to degree $n+1$; we introduce the symbol $x_{n+1}$ and perform $n$
internal brackets and one external bracket:
  \[
  I( [ x_1 x_{n+1} ], x_2, \hdots, x_n ),
  \;
  \hdots,
  \;
  I( x_1, x_2, \hdots, [ x_n x_{n+1} ] ),
  \quad
  [ I( x_1, x_2, \hdots, x_n ), x_{n+1} ].
  \]
If $I$ is a polynomial identity for an algebra $A$ then each lifting is an identity for $A$.
\end{definition}

\begin{definition}
Let $(-,-,-)$ denote a ternary operation which is skew-symmetric in the first and second arguments.
A \textbf{ternary association type} in odd degree $n$ is a valid placement of $(n{-}1)/2$ pairs of
parentheses in $x = x_1 x_2 \cdots x_n$. Each type has a unique factorization $x = (y,z,w)$ where
$n > \deg(y) \ge \deg(z) \ge 1$. The \textbf{standard order} on types is defined analogously to the binary case,
and Definition \ref{definitionbinary1} extends in the obvious way to ternary monomials.
\end{definition}

\begin{definition}
The space $Q_n$ of \textbf{ternary polynomials} in odd degree $n$ is the vector space 
with basis consisting of the ternary monomials in normal form. 
Definition \ref{definitionbinary3} extends in the obvious way to make $Q_n$ into an $S_n$-module.
\end{definition}

\begin{definition}
Let $I(x_1,x_2,\hdots,x_n)$ be a ternary polynomial of odd degree $n$. From $I$ we obtain $n{+}2$
\textbf{liftings} in degree $n{+}2$; we introduce the symbols $x_{n+1}$, $x_{n+2}$ and
perform $n$ internal products and two external products:
  \allowdisplaybreaks \begin{align*}
  &
  I( ( x_1 x_{n+1} x_{n+2} ), x_2, \hdots, x_n ),
  \quad
  \hdots,
  \quad
  I( x_1, x_2, \hdots, ( x_n x_{n+1} x_{n+2} ) ),
  \\
  &
  ( I( x_1, x_2, \hdots, x_n ), x_{n+1}, x_{n+2} ),
  \quad
  ( x_{n+1}, x_{n+2}, I( x_1, x_2, \hdots, x_n ) ).
  \end{align*}
If $I$ is a polynomial identity for a ternary algebra $A$ then so is each lifting.
\end{definition}

We also consider polynomials in which each term can be either a binary monomial, a ternary monomial,
or a mixed monomial in which both operations appear.
The following table gives the number of types and monomials for degrees $\le 7$:
  \begin{center}
  \begin{tabular}{rrrrrr}
  degree & binary & ternary & mixed & b+t+m & monomials \\
  1 &  $\cdot$ & $\cdot$ &  $\cdot$ &   $\cdot$ &      1 \\
  2 &  1 & $\cdot$ &  $\cdot$ &   1 &      1 \\
  3 &  1 & 1 &  $\cdot$ &   2 &      6 \\
  4 &  2 & $\cdot$ &  3 &   5 &     45 \\
  5 &  3 & 2 &  8 &  13 &    510 \\
  6 &  6 & $\cdot$ & 32 &  38 &   7245 \\
  7 & 11 & 6 & 96 & 113 & 126630
  \end{tabular}
  \end{center}
The total number of types is obtained recursively by setting $T_1 = 1$ and using the following
formula, where $\delta = 0$ if $n$ is odd and $\delta = 1$ if $n$ is even:
  \[
  T_n
  =
  \sum_{\tiny \begin{array}{c} i+j=n, \\ i>j \end{array}} 
  \!\!\!
  T_i T_j
  \; + \;
  \delta \, \binom{T_{n/2}}{2}
  \; +
  \!\!\!
  \sum_{\tiny \begin{array}{c} i+j+k=n, \\ i>j \end{array}} 
  \!\!\!
  T_i T_j T_k
  \; + \;
  \delta \, \binom{T_{n/2}}{2} T_k.
  \]
Our computational methods are practical for degrees $\le 7$ with binary or ternary identities,
but only for degrees $\le 6$ with mixed identities, owing to the very large number of mixed monomials
of degree 7.

Two algorithms  are used to determine polynomial identities: 
  \begin{itemize}
  \item
``Fill-and-reduce'' determines a basis for the vector space of multilinear identities of degree $n$ 
satisfied by an algebra $A$.
  \item
``Module generators'' extracts a subset of the basis which
generates the space of identities as an $S_n$-module.
  \end{itemize} 
We describe these algorithms for an algebra with one binary operation; 
the generalization to binary-ternary algebras is straightforward.

\subsection{Fill-and-reduce} 

Let $A$ be a nonassociative algebra of dimension $d$, and let $q$ be the number of multilinear 
monomials in degree $n$. We create a matrix $E$ of size $(q+d) \times q$ 
consisting of a $q \times q$ upper block and a $d \times q$ lower block
and initialize $E$ to zero. We perform the following steps until the rank of $E$ has stabilized; that is, 
the rank has not increased for some fixed number $s$ of iterations:
  \begin{enumerate}
  \item[(1)]
Generate $n$ pseudorandom column vectors $x_1, \hdots, x_n$ of dimension $d$, representing elements
of the algebra $A$.
  \item[(2)]
For each $j = 1, \hdots, q$, evaluate monomial $j$ on the elements $x_1, \hdots,
x_n$ and store the resulting column vector in column $j$ of the lower block.
  \item[(3)]
Compute the row canonical form of $E$; the lower block of $E$ is now zero.
  \end{enumerate}
After the rank has stabilized, the nullspace of $E$ contains the coefficient vectors of the linear
dependence relations on the monomials that are satisfied by many pseudorandom
choices of elements of $A$.  We extract the canonical basis of the nullspace; 
these are the coefficients vectors of nonassociative polynomials which are ``probably'' polynomial
identities satisfied by $A$; these identities still need to be proven directly, or checked
independently by another computation.

\subsection{Module generators} 

In order to reduce the number of polynomial identities, we extract a set of $S_n$-module generators 
from the linear basis of the nullspace computed by the fill-and-reduce algorithm.
Let $I_1, \hdots, I_\ell$ be a basis for the space of multilinear identities in degree $n$ satisfied 
by an algebra $A$. We create a matrix $G$ of size $(q+n!) \times q$ consisting of a $q \times q$ upper block
and an $n! \times q$ lower block and initialize $G$ to zero.
(As before, $q$ is the number of multilinear monomials.)

We set $\mathrm{oldrank} \leftarrow 0$ and then perform the following steps for $k = 1, \hdots, \ell$:
  \begin{enumerate}
  \item[(1)]
Set $i \leftarrow 0$.
  \item[(2)]
For each permutation $\pi$ in $S_n$ do:
    \begin{enumerate}
    \item[(i)]
Set $i \leftarrow i+1$.
    \item[(ii)]
For each $j = 1, \hdots, q$ do:
      \begin{itemize}
      \item
Let $c_j$ be the coefficient in $I_k$ corresponding to monomial $m_j$.
If $c_j \ne 0$ then apply $\pi$ to $m_j$ obtaining $\pi m_j$ and replace $\pi m_j$ by its normal form
$m_{j'}$, keeping track of sign changes.
Store the resulting coefficient $\pm c_j$ in entry $(i,j')$ of the lower block.
      \end{itemize}
    \end{enumerate}
  \item[(3)]
Compute the row canonical form of $G$; the lower block is now zero.
  \item[(4)]
Set $\mathrm{newrank} \leftarrow \mathrm{rank}(G)$.
  \item[(5)]
If $\mathrm{oldrank} < \mathrm{newrank}$ then:
    \begin{enumerate}
    \item[(i)]
Record $I_k$ as a new module generator.
    \item[(ii)]
Set $\mathrm{oldrank} \leftarrow \mathrm{newrank}$.
    \end{enumerate}
  \end{enumerate}
Suppose that we have a set of identities in degree $n$ which generate all the consequences in degree $n$ 
of known identities in degree $< n$.  We apply the module generators algorithm; at termination 
the row space of $G$ contains a basis for the subspace of
identities which are consequences of identities of lower degree. Without changing $G$,
we then apply the
module generators algorithm to the linear basis produced by the fill-and-reduce algorithm. 
We obtain a set of module generators for the space of all identities in
degree $n$ modulo the space of known identities in degree $n$; that is, a set of module generators
for the new identities in degree $n$.

\subsection{Rational arithmetic and modular arithmetic}

We usually assume that the algebra $A$ is defined over the field $\mathbb{Q}$ of rational numbers;
but in order to save computer time and memory we often use modular arithmetic.
The structure theory of the group algebra $\mathbb{Q} S_n$ shows that the primitive 
orthogonal idempotents, and hence the matrix units in the Wedderburn decomposition, all have 
coefficients whose denominators are divisors of $n!$; see Clifton \cite{Clifton}.
Hence computations with modular arithmetic for any prime $p > n$ will provide the ``same''
results: in particular, the ranks of the matrices produced by the two algorithms will be equal.

Suppose that $I$ is an identity with rational coefficients satisfied by the algebra $A$ which
has integral structure constants with respect to a given basis. We multiply $I$ by the least common
multiple of the denominators of its coefficients, obtaining an identity $I'$ with integral
coefficients; we then divide $I'$ by the greatest common divisor of its coefficients, obtaining an
identity $I''$ which is primitive in the sense that its coefficients are integers with no common
factor. It is clear that $I''$ is a polynomial identity satisfied by the algebra $A$, and that the
reduction of $I''$ modulo $p$ is nonzero for any prime number $p$. Therefore, existence of identities
in characteristic 0 implies existence in characteristic $p$ for all $p$, and so
non-existence in characteristic $p$ for any $p$ implies non-existence in characteristic 0.

For a more detailed discussion of these issues, see Bremner and Peresi \cite[\S 5]{BremnerPeresi2}.


\section{Polynomial identities for the bilinear operation} \label{sectionbinary}

We study the identities of degree $\le 7$ for the simple non-Lie Malcev algebra.

\begin{lemma} \label{degree4lemma}
Every polynomial identity of degree $\le 4$ satisfied by the bilinear operation $[x,y] = \alpha( x \wedge y)$ on the 
$\s$-module $V(6)$ is a consequence of the defining identities for Malcev algebras from Definition \ref{definitionMalcev}.
\end{lemma}

\begin{proof}
It is easy to check that every identity of degree $\le 3$ satisfied by $[x,y]$ is a consequence of
anticommutativity.  In degree 4, there are two association types and 15 multilinear monomials 
for an anticommutative operation:
  \[
  \begin{array}{llllll}
  {[[[a{,}b]{,}c]{,}d]}{,} &
  [[[a{,}b]{,}d]{,}c]{,} &
  [[[a{,}c]{,}b]{,}d]{,} &
  [[[a{,}c]{,}d]{,}b]{,} &
  [[[a{,}d]{,}b]{,}c]{,} &
  [[[a{,}d]{,}c]{,}b]{,}
  \\
  {[[[b{,}c]{,}a]{,}d]}{,} &
  [[[b{,}c]{,}d]{,}a]{,} &
  [[[b{,}d]{,}a]{,}c]{,} &
  [[[b{,}d]{,}c]{,}a]{,} &
  [[[c{,}d]{,}a]{,}b]{,} &
  [[[c{,}d]{,}b]{,}a]{,} 
  \\
  {[[a{,}b]{,}[c{,}d]]}{,} &
  [[a{,}c]{,}[b{,}d]]{,} &
  [[a{,}d]{,}[b{,}c]].
  \end{array}
  \]  
We create a $22 \times 15$ matrix with a $15 \times 15$ upper block
and a $7 \times 15$ lower block.  
During each iteration, we perform the following fill-and-reduce calculation:
  \begin{enumerate}
  \item
generate four pseudorandom vectors of dimension 7 with integer components of absolute value $\le 9$;
  \item
assign these vectors to the variables $a,b,c,d$ and evaluate each of the 15 monomials using the
structure constants of Table \ref{alphatable};
  \item
for each $j$, store the $7 \times 1$ column vector obtained from monomial $j$ in column $j$ of the 
lower block;
  \item
compute the row canonical form (RCF) of the matrix.
  \end{enumerate}
After the second iteration, the rank of the matrix is 10, and this value does not increase during 
another 100 iterations.  The nonzero rows are as follows:
  \[
  \left[
  \begin{array}{rrrrrrrrrrrrrrr}
  1 & . & . & . & . & . & . & . & . &  . & . &  1 &  . &  1 &  . \\
  . & 1 & . & . & . & . & . & . & . &  . & . & -1 &  . &  . &  1 \\
  . & . & 1 & . & . & . & . & . & . &  1 & . &  . &  1 &  . &  . \\
  . & . & . & 1 & . & . & . & . & . & -1 & . &  . &  . &  . & -1 \\
  . & . & . & . & 1 & . & . & . & . &  1 & . & -1 &  . & -1 &  1 \\
  . & . & . & . & . & 1 & . & . & . & -1 & . &  1 & -1 &  . & -1 \\
  . & . & . & . & . & . & 1 & . & . &  1 & . & -1 &  . &  . &  1 \\
  . & . & . & . & . & . & . & 1 & . & -1 & . &  1 & -1 &  1 & -1 \\
  . & . & . & . & . & . & . & . & 1 &  1 & . &  . &  1 &  . &  1 \\
  . & . & . & . & . & . & . & . & . &  . & 1 &  1 &  . &  1 & -1
  \end{array}
  \right]
  \]
The canonical basis of the nullspace consists of the rows of the following matrix, which have been sorted
by increasing number of nonzero components:
  \[
  \left[
  \begin{array}{rrrrrrrrrrrrrrr}
  -1 &  . &  . & . &  1 &  . &  . & -1 &  . & . & -1 & . & . & 1 & . \\
   . &  . & -1 & . &  . &  1 &  . &  1 & -1 & . &  . & . & 1 & . & . \\
  -1 &  1 &  . & . &  1 & -1 &  1 & -1 &  . & . & -1 & 1 & . & . & . \\
   . &  . & -1 & 1 & -1 &  1 & -1 &  1 & -1 & 1 &  . & . & . & . & . \\
   . & -1 &  . & 1 & -1 &  1 & -1 &  1 & -1 & . &  1 & . & . & . & 1
  \end{array}
  \right]
  \]
The first row is the coefficient vector of this polynomial identity:
  \begin{equation}
  \label{m-identity}
  {} 
  - [[a,b],c],d] 
  + [[a,d],b],c] 
  - [[b,c],d],a] 
  - [[c,d],a],b] 
  + [a,c],[b,d]]; 
  \end{equation}
this is equivalent to the second identity of Definition \ref{definitionMalcev}.  We apply all 24 permutations 
of $a,b,c,d$ to this identity, and verify that the resulting $24 \times 15$ matrix of coefficient vectors has 
the same row space as the preceding $5 \times 15$ matrix.
\end{proof}

For notational convenience, we omit the brackets for polynomials of degree $\ge 5$ in an anticommutative
operation, and write the product as juxtaposition.

\begin{definition} \label{h-polynomial}
(Filippov \cite{Filippov}, Elduque \cite{Elduque})
The \textbf{Filippov $h$-polynomial} is the following nonassociative polynomial in any anticommutative algebra:
  \begin{align*}
  h(a,b,c,d,e)
  &= 
     (((ab)c)d)e 
  +  (((ab)c)e)d 
  -  (((ab)d)c)e 
  -  (((ab)e)c)d
  \\
  &\quad
  +  (((ad)b)e)c 
  -  (((ad)e)b)c
  +  (((ae)b)d)c 
  -  (((ae)d)b)c
  \\
  &\quad 
  + 2((ab)(cd))e
  + 2((ab)(ce))d 
  + 2((ad)(be))c 
  + 2((ae)(bd))c.
  \end{align*}
\end{definition}

\begin{definition} \label{18termidentity}
The \textbf{18-term identity} is the following nonassociative polynomial in any anticommutative algebra:
  \begin{align*}
  k(a,b,c,d,e)
  &=
  - (((ab)c)e)d 
  + (((ab)e)d)c 
  + (((ac)b)d)e 
  - (((ac)d)e)b 
  \\
  &\quad
  - (((ad)b)c)e 
  - (((ad)c)b)e 
  + (((ae)b)c)d 
  + (((ae)c)b)d 
  \\
  &\quad
  + (((bd)c)a)e 
  - (((be)a)d)c
  + (((be)c)d)a 
  - (((be)d)a)c
  \\
  &\quad 
  + (((cd)a)e)b 
  - (((cd)b)e)a 
  + (((cd)e)a)b 
  - (((ce)b)a)d 
  \\
  &\quad
  + ((ab)c)(de) 
  + ((ac)b)(de). 
  \end{align*}
\end{definition}

\begin{lemma}
A Malcev algebra satisfies the identity $h(a,b,c,d,e) \equiv 0$ if and only if 
it satisfies the identity $k(a,b,c,d,e) \equiv 0$.
\end{lemma}

\begin{proof}
The following equations hold modulo the consequences in degree 5 of the 
defining identities for Malcev algebras from Definition \ref{definitionMalcev}:
  \begin{align*}
  h(a,b,c,d,e)
  &\equiv
  {}
  - k(a,b,c,e,d) 
  - k(a,b,d,e,c) 
  + k(a,b,e,d,c)
  \\
  &\qquad 
  + k(a,c,d,b,e) 
  + k(a,c,e,d,b) 
  - k(b,a,c,d,e),
  \\
  2 \, k(a,b,c,d,e)
  &\equiv
  {} 
  - h(a,b,c,d,e) 
  + h(a,c,d,b,e) 
  + h(a,d,b,c,e) 
  \\
  &\qquad
  - h(b,c,a,d,e) 
  - h(b,d,a,c,e) 
  - h(c,d,a,b,e).
  \end{align*}
In each case, the difference between the two sides is a linear combination
of permutations of $m(ae,b,c,d)$ and $m(a,b,c,d)e$ where $m$ is given by equation \eqref{m-identity}.
\end{proof}

\begin{remark}
The polynomial $h(a,b,c,d,e)$ has 12 terms, but the sum of the squares of the coefficients is 24.  
The polynomial $k(a,b,c,d,e)$ has 18 terms, but the sum of the squares of the coefficients is 18.
In this sense, the 18-term identity is simpler than the Filippov $h$-polynomial.
\end{remark}

\begin{proposition} \label{degree5proposition}
Every polynomial identity of degree $\le 5$ satisfied by the bilinear operation $[x,y] = \alpha( x \wedge y)$ 
on the $\s$-module $V(6)$ is a consequence of the defining identities for Malcev algebras and the 18-term
identity of degree 5.
\end{proposition}

\begin{proof}
There are three anticommutative association types in degree 5,
  \[
  ( ( ( - - ) - ) - ) -, 
  \qquad
  ( ( - - ) ( - - ) ) -, 
  \qquad
  ( ( - - ) - ) ( - - ),
  \]
with respectively 60, 15, 30 multilinear monomials for a total of 105.  Let $m$ be 
the identity of equation \eqref{m-identity}; every consequence of $m$ in degree 5 follows from 
$m(ae,b,c,d)$ and $m(a,b,c,d)e$.  The permutations of these two identities are represented by the
row space of a $240 \times 120$ matrix; the rank of this matrix is 61. 

We proceed as in the proof of Lemma \ref{degree4lemma}.  We create a matrix with a $105 \times 105$ upper
block and a $7 \times 105$ lower block.  The fill-and-reduce algorithm achieves rank 34, and this
remains constant for another 100 iterations.  The nullspace has dimension 71, and hence
there is a 10-dimensional space of new identities in degree 5 that
do not follow from the known identity $m$ in degree 4.  

To find the simplest identity which generates the new
identities, we use the Hermite normal form of an integer matrix together with the LLL algorithm for lattice
basis reduction as in Bremner and Peresi \cite{BremnerPeresi1}; see also Bremner \cite{Bremner}.  
We redo the fill-and-reduce calculation, but instead 
of computing the RCF after each iteration, we compute the HNF to ensure that the matrix entries remain 
integers; we obtain a $34 \times 105$ matrix $M$.  We compute the 
Hermite normal form $H$ of the transpose $M^t$, and a $105 \times 105$ matrix $U$ for 
which $U M^t = H$.  The last 71 rows of $U$ form a lattice basis for the integer nullspace of $M$.  We apply the LLL
algorithm to this basis, and sort the resulting vectors by increasing Euclidean length; these vectors are the
coefficient vectors of polynomial identities satisfied by $\alpha( x \wedge y )$.  
Further calculations show that the entire nullspace is spanned by the permutations of a single identity;
this is the 18-term identity of Definition \ref{18termidentity}.
\end{proof}

\begin{theorem} \label{bilineardegree67}
Every polynomial identity of degree $\le 7$ satisfied by the bilinear operation $[x,y] = \alpha( x \wedge y)$ 
on the $\s$-module $V(6)$ is a consequence of the defining identities for Malcev algebras and the 18-term
identity of degree 5.
\end{theorem}

\begin{proof}
The computations are similar to those for Lemma \ref{degree4lemma} and
Proposition \ref{degree5proposition}, except the matrices are larger;
we used modular arithmetic with $p = 101$.   

For degree 6, there are six anticommutative association types:
  \begin{alignat*}{3}
  &
  ( ( ( ( - - ) - ) - ) - ) -,
  &\qquad
  &
  ( ( ( - - ) ( - - ) ) - ) -,
  &\qquad
  &
  ( ( ( - - ) - ) ( - - ) ) -,
  \\
  &
  ( ( ( - - ) - ) - ) ( - - ),
  &\qquad
  &
  ( ( - - ) ( - - ) ) ( - - ),
  &\qquad
  &
  ( ( - - ) - ) ( ( - - ) - ),
  \end{alignat*}
with $360 + 90 + 180 + 180 + 45 + 90 = 945$ multilinear monomials.
Fill-and-reduce stabilizes at rank 120, indicating a nullspace of dimension 825.
Every identity in degree 6 which is a consequence of the identities of lower degree follows
from anticommutativity and the following seven identities, where $m$ and $k$ are
given by equation \eqref{m-identity} and Definition \ref{18termidentity}:
  \begin{alignat*}{4}
  &
  m((af)e,b,c,d),
  &\qquad
  &
  m(ae,bf,c,d),
  &\qquad
  &
  m(ae,b,cf,d),
  &\qquad
  &
  m(ae,b,c,d)f,
  \\
  &
  (m(a,b,c,d)e)f,
  &\qquad
  &
  k(af,b,c,d,e),
  &\qquad
  &
  k(a,bf,c,d,e).
  \end{alignat*}
These identities generate a subspace of dimension 825, which is contained in the nullspace of 
the fill-and-reduce matrix, so there are no new identities in degree 6.

For degree 7, there are eleven anticommutative association types, 
  \begin{alignat*}{3}
  &
  ( ( ( ( ( - - ) - ) - ) - ) - ) -, 
  &\qquad
  &
  ( ( ( ( - - ) ( - - ) ) - ) - ) -, 
  &\qquad
  &
  ( ( ( ( - - ) - ) ( - - ) ) - ) -, 
  \\
  &
  ( ( ( ( - - ) - ) - ) ( - - ) ) -, 
  &\qquad
  &
  ( ( ( - - ) ( - - ) ) ( - - ) ) -, 
  &\qquad
  &
  ( ( ( - - ) - ) ( ( - - ) - ) ) -, 
  \\
  &
  ( ( ( ( - - ) - ) - ) - ) ( - - ),
  &\qquad
  &
  ( ( ( - - ) ( - - ) ) - ) ( - - ),
  &\qquad
  &
  ( ( ( - - ) - ) ( - - ) ) ( - - ),
  \\
  &
  ( ( ( - - ) - ) - ) ( ( - - ) - ),
  &\qquad
  &
  ( ( - - ) ( - - ) ) ( ( - - ) - ),
  \end{alignat*}
with $2520 + 630 + 1260 + 1260 + 315 + 630 + 1260 + 315 + 630 + 1260 + 315 = 10395$
multilinear monomials.
Fill-and-reduce stabilizes at rank 454, indicating a nullspace of dimension 9941.
Further computations show that the consequences of the identities of lower degree generate the entire nullspace.
\end{proof}

We mention the following classical conjecture in the theory of Malcev algebras.

\begin{conjecture} \label{filippovconjecture}
Every polynomial identity satisfied by the simple 7-dimensional Malcev algebra is a consequence of the defining 
identities for Malcev algebras and the Filippov $h$-polynomial of degree 5.
\end{conjecture}


\section{Polynomial identities for the trilinear operation} \label{sectionternary}

The operation $(x,y,z) = \beta( x \wedge y ) \cdot z$ is the composition of the $\s$-module morphism $\beta$ 
of Table \ref{betatable} with the action of $\s$ on $V(6)$ in equation \eqref{HEFactionV(6)}; recall
the sign change for $E$ in equation \eqref{signchange}.  
The structure constants appear in Table \ref{trilineartable}.  

\begin{lemma}
Every polynomial identity of degree $\le 3$ for the trilinear operation $(x,y,z) = \beta( x \wedge y ) \cdot z$ 
on the $\s$-module $V(6)$ is a consequence of the skew-symmetry in the first two arguments:
$(a,b,c) + (b,a,c) \equiv 0$.
\end{lemma}

\begin{proof}
Elementary.
\end{proof}

  \begin{table}
  \begin{align*}
  &
  \begin{array}{r|rrrrrrr} 
  & v_6 & v_4 & v_2 & v_0 & v_{-2} & v_{-4} & v_{-6} 
  \\
  \midrule
  v_{ 6} &    . &    . &    . &    . &    . &    . &    6 \\[-2pt] 
  v_{ 4} &    . &    . &    . &    . &    . &  -24 &    2 \\[-2pt] 
  v_{ 2} &    . &    . &    . &    . &   30 &  -10 &    . \\[-2pt] 
  (-,-,v_{ 6}) \qquad 
  v_{ 0} &    . &    . &    . &    . &   20 &    . &    . \\[-2pt] 
  v_{-2} &    . &    . &  -30 &  -20 &    . &    . &    . \\[-2pt] 
  v_{-4} &    . &   24 &   10 &    . &    . &    . &    . \\[-2pt] 
  v_{-6} &   -6 &   -2 &    . &    . &    . &    . &    . 
  \\ 
  \midrule 
  v_{ 6} &    . &    . &    . &    . &    . &  -12 &    4 \\[-2pt] 
  v_{ 4} &    . &    . &    . &    . &   60 &  -16 &    4 \\[-2pt] 
  v_{ 2} &    . &    . &    . & -120 &   20 &  -20 &    . \\[-2pt] 
  (-,-,v_{ 4}) \qquad 
  v_{ 0} &    . &    . &  120 &    . &   40 &    . &    . \\[-2pt] 
  v_{-2} &    . &  -60 &  -20 &  -40 &    . &    . &    . \\[-2pt] 
  v_{-4} &   12 &   16 &   20 &    . &    . &    . &    . \\[-2pt] 
  v_{-6} &   -4 &   -4 &    . &    . &    . &    . &    . 
  \\ 
  \midrule 
  v_{ 6} &    . &    . &    . &    . &    . &  -10 &    2 \\[-2pt] 
  v_{ 4} &    . &    . &    . &    . &   50 &   -8 &    6 \\[-2pt] 
  v_{ 2} &    . &    . &    . & -100 &   10 &  -30 &    . \\[-2pt] 
  (-,-,v_{ 2}) \qquad 
  v_{ 0} &    . &    . &  100 &    . &   60 &    . &    . \\[-2pt] 
  v_{-2} &    . &  -50 &  -10 &  -60 &    . &    . &    . \\[-2pt] 
  v_{-4} &   10 &    8 &   30 &    . &    . &    . &    . \\[-2pt] 
  v_{-6} &   -2 &   -6 &    . &    . &    . &    . &    . 
  \\ 
  \midrule 
  v_{ 6} &    . &    . &    . &    . &    . &   -8 &    . \\[-2pt] 
  v_{ 4} &    . &    . &    . &    . &   40 &    . &    8 \\[-2pt] 
  v_{ 2} &    . &    . &    . &  -80 &    . &  -40 &    . \\[-2pt] 
  (-,-,v_{ 0}) \qquad 
  v_{ 0} &    . &    . &   80 &    . &   80 &    . &    . \\[-2pt] 
  v_{-2} &    . &  -40 &    . &  -80 &    . &    . &    . \\[-2pt] 
  v_{-4} &    8 &    . &   40 &    . &    . &    . &    . \\[-2pt] 
  v_{-6} &    . &   -8 &    . &    . &    . &    . &    . 
  \\ 
  \midrule 
  v_{ 6} &    . &    . &    . &    . &    . &   -6 &   -2 \\[-2pt] 
  v_{ 4} &    . &    . &    . &    . &   30 &    8 &   10 \\[-2pt] 
  v_{ 2} &    . &    . &    . &  -60 &  -10 &  -50 &    . \\[-2pt] 
  (-,-,v_{-2}) \qquad 
  v_{ 0} &    . &    . &   60 &    . &  100 &    . &    . \\[-2pt] 
  v_{-2} &    . &  -30 &   10 & -100 &    . &    . &    . \\[-2pt] 
  v_{-4} &    6 &   -8 &   50 &    . &    . &    . &    . \\[-2pt] 
  v_{-6} &    2 &  -10 &    . &    . &    . &    . &    . 
  \\ 
  \midrule 
  v_{ 6} &    . &    . &    . &    . &    . &   -4 &   -4 \\[-2pt] 
  v_{ 4} &    . &    . &    . &    . &   20 &   16 &   12 \\[-2pt] 
  v_{ 2} &    . &    . &    . &  -40 &  -20 &  -60 &    . \\[-2pt] 
  (-,-,v_{-4}) \qquad 
  v_{ 0} &    . &    . &   40 &    . &  120 &    . &    . \\[-2pt] 
  v_{-2} &    . &  -20 &   20 & -120 &    . &    . &    . \\[-2pt] 
  v_{-4} &    4 &  -16 &   60 &    . &    . &    . &    . \\[-2pt] 
  v_{-6} &    4 &  -12 &    . &    . &    . &    . &    . 
  \\ 
  \midrule 
  v_{ 6} &    . &    . &    . &    . &    . &   -2 &   -6 \\[-2pt] 
  v_{ 4} &    . &    . &    . &    . &   10 &   24 &    . \\[-2pt] 
  v_{ 2} &    . &    . &    . &  -20 &  -30 &    . &    . \\[-2pt] 
  (-,-,v_{-6}) \qquad 
  v_{ 0} &    . &    . &   20 &    . &    . &    . &    . \\[-2pt] 
  v_{-2} &    . &  -10 &   30 &    . &    . &    . &    . \\[-2pt] 
  v_{-4} &    2 &  -24 &    . &    . &    . &    . &    . \\[-2pt] 
  v_{-6} &    6 &    . &    . &    . &    . &    . &    .   
  \end{array}  
  \end{align*}
  \caption{Structure constants for the trilinear operation}
  \label{trilineartable}
  \end{table}
  
  \begin{table}
  \begin{align*}
  &
  \big[
    (((bce)da)fg) 
  + (((bdc)ea)fg) 
  + (((bec)da)fg) 
  \big]
  \\[-3pt]
  -
  &\big[
    (((bcd)ea)fg) 
  + (((bde)ca)fg) 
  + (((bed)ca)fg) 
  \big]
  \\[-3pt]
  +
  2
  &\big[
    ((ab(cde))fg) 
  + ((ab(dec))fg) 
  + ((ac(bed))fg)
  + ((ad(bce))fg) 
  \\[-3pt]
  &\; 
  + ((ad(ceb))fg) 
  + ((ae(bdc))fg) 
  + ((bc(dea))fg) 
  + ((cd(bea))fg) 
  \big]
  \\[-3pt]
  -
  2
  &\big[
    ((ab(ced))fg) 
  + ((ac(bde))fg) 
  + ((ac(deb))fg) 
  + ((ad(bec))fg) 
  \\[-3pt]
  &\; 
  + ((ae(bcd))fg) 
  + ((ae(cdb))fg) 
  + ((bd(cea))fg) 
  \big]
  \\[-3pt]
  +
  3
  &\big[
    (((abc)de)fg) 
  + (((abd)ec)fg) 
  + (((abe)cd)fg) 
  + (((acb)ed)fg) 
  \\[-3pt]
  &\; 
  + (((acd)be)fg) 
  + (((ace)db)fg) 
  + (((adb)ce)fg) 
  + (((adc)eb)fg) 
  \\[-3pt]
  &\; 
  + (((ade)bc)fg) 
  + (((aeb)dc)fg) 
  + (((aec)bd)fg) 
  + (((aed)cb)fg) 
  \big]
  \\[-3pt]
  -
  3
  &\big[
    (((abc)ed)fg) 
  + (((abd)ce)fg) 
  + (((abe)dc)fg) 
  + (((acb)de)fg) 
  \\[-3pt]
  &\; 
  + (((acd)eb)fg) 
  + (((ace)bd)fg) 
  + (((adb)ec)fg) 
  + (((adc)be)fg) 
  \\[-3pt]
  &\; 
  + (((ade)cb)fg) 
  + (((aeb)cd)fg) 
  + (((aec)db)fg) 
  + (((aed)bc)fg) 
  \big]
  \\[-3pt]
  +
  4  
  &\big[
    (((abc)ef)dg) 
  + (((abd)cf)eg) 
  + (((abe)df)cg) 
  + (((acb)df)eg) 
  \\[-3pt]
  &\; 
  + (((acd)ef)bg) 
  + (((ace)bf)dg) 
  + (((adb)ef)cg) 
  + (((adc)bf)eg) 
  \\[-3pt]
  &\; 
  + (((ade)cf)bg) 
  + (((aeb)cf)dg) 
  + (((aec)df)bg) 
  + (((aed)bf)cg) 
  \\[-3pt]
  &\; 
  + (((afb)cd)eg) 
  + (((afb)de)cg) 
  + (((afb)ec)dg) 
  + (((afc)be)dg) 
  \\[-3pt]
  &\; 
  + (((afc)db)eg) 
  + (((afc)ed)bg) 
  + (((afd)bc)eg) 
  + (((afd)ce)bg) 
  \\[-3pt]
  &\; 
  + (((afd)eb)cg) 
  + (((afe)bd)cg) 
  + (((afe)cb)dg) 
  + (((afe)dc)bg) 
  \big]
  \\[-3pt]
  -
  4  
  &\big[
    (((abc)df)eg) 
  + (((abd)ef)cg) 
  + (((abe)cf)dg) 
  + (((acb)ef)dg) 
  \\[-3pt]
  &\; 
  + (((acd)bf)eg) 
  + (((ace)df)bg) 
  + (((adb)cf)eg) 
  + (((adc)ef)bg) 
  \\[-3pt]
  &\; 
  + (((ade)bf)cg) 
  + (((aeb)df)cg) 
  + (((aec)bf)dg) 
  + (((aed)cf)bg) 
  \\[-3pt]
  &\; 
  + (((afb)ce)dg) 
  + (((afb)dc)eg) 
  + (((afb)ed)cg) 
  + (((afc)bd)eg) 
  \\[-3pt]
  &\; 
  + (((afc)de)bg) 
  + (((afc)eb)dg) 
  + (((afd)be)cg) 
  + (((afd)cb)eg) 
  \\[-3pt]
  &\; 
  + (((afd)ec)bg) 
  + (((afe)bc)dg) 
  + (((afe)cd)bg) 
  + (((afe)db)cg) 
  \big]
  \\[-3pt]
  +
  8  
  &\big[
    (((abc)fe)dg) 
  + (((abd)fc)eg) 
  + (((abe)fd)cg) 
  + (((abf)cd)eg) 
  \\[-3pt]
  &\; 
  + (((abf)de)cg) 
  + (((abf)ec)dg) 
  + (((acb)fd)eg) 
  + (((acd)fe)bg) 
  \\[-3pt]
  &\; 
  + (((ace)fb)dg) 
  + (((acf)be)dg) 
  + (((acf)db)eg) 
  + (((acf)ed)bg) 
  \\[-3pt]
  &\; 
  + (((adb)fe)cg) 
  + (((adc)fb)eg) 
  + (((ade)fc)bg) 
  + (((adf)bc)eg) 
  \\[-3pt]
  &\; 
  + (((adf)ce)bg) 
  + (((adf)eb)cg) 
  + (((aeb)fc)dg) 
  + (((aec)fd)bg) 
  \\[-3pt]
  &\; 
  + (((aed)fb)cg) 
  + (((aef)bd)cg) 
  + (((aef)cb)dg) 
  + (((aef)dc)bg) 
  \big]
  \\[-3pt]
  -
  8
  &\big[
    (((abc)fd)eg) 
  + (((abd)fe)cg) 
  + (((abe)fc)dg) 
  + (((abf)ce)dg) 
  \\[-3pt]
  &\; 
  + (((abf)dc)eg) 
  + (((abf)ed)cg) 
  + (((acb)fe)dg) 
  + (((acd)fb)eg) 
  \\[-3pt]
  &\; 
  + (((ace)fd)bg) 
  + (((acf)bd)eg) 
  + (((acf)de)bg) 
  + (((acf)eb)dg) 
  \\[-3pt]
  &\; 
  + (((adb)fc)eg) 
  + (((adc)fe)bg) 
  + (((ade)fb)cg) 
  + (((adf)be)cg) 
  \\[-3pt]
  &\; 
  + (((adf)cb)eg) 
  + (((adf)ec)bg) 
  + (((aeb)fd)cg) 
  + (((aec)fb)dg) 
  \\[-3pt]
  &\; 
  + (((aed)fc)bg) 
  + (((aef)bc)dg) 
  + (((aef)cd)bg) 
  + (((aef)db)cg) 
  \big]
  \end{align*}
  \caption{A new identity for the trilinear operation in degree 7}
  \label{degree7ternary}
  \end{table}
  
\begin{proposition} \label{ternarydegree5}
Every polynomial identity of degree $\le 5$ for the trilinear operation $(x,y,z) = \beta( x \wedge y ) \cdot z$ 
on the $\s$-module $V(6)$ is a consequence of skew-symmetry and the
ternary derivation identity: 
  \[
  ((a,b,c),d,e) - ((a,b,d),c,e) - (a,b,(c,d,e)) + (c,d,(a,b,e)) \equiv 0.
  \]
\end{proposition}

\begin{proof}
There are two association types for a skew-symmetric trilinear operation in degree 5,
$((-,-,-),-,-)$ and $(-,-,(-,-,-))$, with 60 and 30 multilinear monomials,
for a total of 90.  The fill-and-reduce process stabilizes at rank 60, indicating a
nullspace of dimension 30.  We compute the canonical basis of the nullspace;
one vector generates the entire nullspace as an $S_5$-module.  
This is the coefficient vector of the ternary derivation identity.
\end{proof}

\begin{theorem}
For the trilinear operation $(x,y,z) = \beta( x \wedge y ) \cdot z$ on the $\s$-module $V(6)$, 
there is a 357-dimensional space of multilinear polynomial identities in degree 7 which are not 
consequences of skew-symmetry and the ternary derivation identity.  One such identity is displayed 
in Table \ref{degree7ternary}; this 141-term identity generates a 42-dimensional subspace of the 
space of new identities in degree 7.
\end{theorem}

\begin{proof}
We use modular arithmetic with $p = 101$ to find the new identity, 
but we will check the new identity using rational arithmetic.
There are six association types in degree 7 for a skew-symmetric trilinear operation:
  \begin{alignat*}{3}
  &
  (((-,-,-),-,-),-,-),
  &\quad
  &
  ((-,-,(-,-,-)),-,-),
  &\quad
  &
  (-,-,((-,-,-),-,-)),
  \\
  &
  (-,-,(-,-,(-,-,-))),
  &\quad
  &
  ((-,-,-),(-,-,-),-),
  &\quad
  &
  ((-,-,-),-,(-,-,-)),
  \end{alignat*}
with respectively $2520 + 1260 + 1260 + 630 + 630 + 1260 = 7560$ multilinear monomials.
Fill-and-reduce stabilizes at rank 2793, indicating a nullspace of dimension 4767.

We have to exclude the consequences of the ternary derivation identity.  Given a
multilinear ternary polynomial $t(a,b,c,d,e)$ of degree 5, we have the following eight generators 
for the $S_7$-module of multilinear consequences of $t$:
  \begin{alignat*}{4}
  &
  t( (a,f,g), b, c, d, e ),
  &\;\;
  &
  t( a, (b,f,g), c, d, e ),
  &\;\;
  &
  t( a, b, (c,f,g), d, e ),
  &\;\;
  &
  t( a, b, c, (d,f,g), e ),
  \\
  &
  t( a, b, c, d, (e,f,g) ),
  &\;\;
  &
  ( t( a, b, c, d, e ), f, g ),
  &\;\;
  &
  ( f, t( a, b, c, d, e ), g ),
  &\;\;
  &
  ( f, g, t( a, b, c, d, e ) ).
  \end{alignat*}
For the ternary derivation identity, these generators produce a subspace of dimension 4410, 
indicating a complementary subspace of 
dimension $4767 - 4410 = 357$ consisting of new identities for the trilinear operation.

We compute the canonical basis of the nullspace of the matrix obtained from the fill-and-reduce process.
We sort these 4767 vectors by increasing number of nonzero components;
that is, by increasing number of terms in the corresponding polynomial identities.
The first 3016 vectors lie in the 4410-dimensional subspace generated by the consequences of the ternary 
derivation identity.
Vector 3017 is the coefficient vector of the identity in Table \ref{degree7ternary}; it increases the
dimension of the subspace to 4452.
Modulo 101, the nonzero components of this vector are 
$1, 2, 4, 49, 50, 51, 52, 97, 99, 100$;
multiplying by 2 and reducing modulo 101 using symmetric representatives, we obtain 
$2, 4, 8, -3, -1, 1, 3, -8, -4, -2$.
We use these integer coefficients to check the validity of the identity: we generate 7 pseudorandom 
7-dimensional vectors with single-digit components, set the variables $a,b,c,d,e,f,g$ equal to these
vectors, evaluate the identity using integer arithmetic, and obtain the zero vector.  After performing
1000 independent tests, we are convinced that the identity in Table \ref{degree7ternary} is valid over a
field of characteristic 0.
\end{proof}

\begin{problem}
Find a minimal set of $S_7$-module generators for the space of multilinear polynomial identities
of degree 7 satisfied by the operation $(x,y,z)$. 
\end{problem}


\section{Polynomial identities relating the two operations} \label{sectionmixed}

We now consider polynomial identities which involve both operations.
To obtain the correct identities, the scalar factors in the morphisms $\alpha$ and $\beta$ must be compatible.  
We therefore replace $\beta$ by $-30 \beta$; that is, we replace $(x,y,z)$ by $-30(x,y,z)$.
This does not affect the previous results: polynomial identities for a single operation
do not change if the operation is replaced by a (nonzero) scalar multiple.

\begin{lemma} \label{mixeddegree3}
Every polynomial identity of degree $\le 3$ relating the bilinear operation $[x,y] = \alpha( x \wedge y)$
and the trilinear operation $(x,y,z) = \beta( x \wedge y ) \cdot z$ on the $\s$-module $V(6)$ is a consequence 
of the anticommutativity of $[x,y]$, the skew-symmetry of $(x,y,z)$, and the mixed Jacobi identity in degree 3
from Definition \ref{definitionLY}:
  \[
  [[a,b],c] + [[b,c],a] + [[c,a],b] + (a,b,c) + (b,c,a) + (c,a,b)
  \equiv 0.
  \]
\end{lemma}

\begin{proof}
There are two association types in degree 3, $[[-,-],-]$ and $(-,-,-)$, and each has three multilinear monomials,
all of which appear in the mixed Jacobi identity.
Fill-and-reduce stabilizes at rank 5, with this row canonical form:
  \[
  \left[
  \begin{array}{rrrrrr}
  1 & . & . & . & . & -1 \\
  . & 1 & . & . & . &  1 \\
  . & . & 1 & . & . & -1 \\
  . & . & . & 1 & . & -1 \\
  . & . & . & . & 1 &  1
  \end{array}
  \right]
  \]
The canonical basis vector of the nullspace is the coefficient vector for the mixed Jacobi identity.
(This explains why rescaled the trilinear operation.)
\end{proof}

\begin{proposition} \label{mixeddegree4}
Every polynomial identity of degree $\le 4$ relating the bilinear operation $[x,y] = \alpha( x \wedge y)$
and the trilinear operation $(x,y,z) = \beta( x \wedge y ) \cdot z$ on the $\s$-module $V(6)$ is a consequence 
of the anticommutativity of $[x,y]$, the skew-symmetry of $(x,y,z)$, the mixed Jacobi identity in degree 3,
the Malcev identity in degree 4, and these two identities in degree 4 from Definition \ref{definitionLY}:
  \begin{align*}
  &
  ( [a,b], c, d ) + ( [b,c], a, d ) + ( [c,a], b, d )
  \equiv 0,
  \\
  &
  [ (a,b,c), d ] - [ (a,b,d), c ] - (a,b,[c,d])
  \equiv 0.
  \end{align*}
\end{proposition}

\begin{proof}
In degree 4, there are five association types involving both operations,
  \[
  [ [ [ -, - ], - ], - ], \quad 
  [ ( -, -, - ), - ], \quad 
  [ [ -, - ], [ -, - ] ], \quad 
  ( [ -, - ], -, - ), \quad 
  ( -, -, [ -, - ] ), 
  \]
with $12 + 12 + 3 + 12 + 6 = 45$ multilinear monomials.
Fill-and-reduce stabilizes at rank 21 indicating a nullspace of dimension 24.
We compute the canonical basis of the nullspace and sort the vectors by increasing number of nonzero components.

The mixed Jacobi identity $j(a,b,c)$ has these four consequences in degree 4:
  \[
  j( [a,d], b, c ), \quad
  j( a, [b,d], c ), \quad
  j( a, b, [c,d] ), \quad
  [ j(a,b,c), d ].
  \]
Further calculations show that these identities generate a 10-dimensional subspace of the nullspace from 
the fill-and-reduce process. 
Hence there exists a complementary subspace of dimension $24 - 10 = 14$ consisting of new identities in degree 4.
Processing the sorted list of nullspace basis vectors, we find that three vectors increase the rank;
these are the coefficient vectors of the Malcev identity and the two identities
in degree 4 from Definition \ref{definitionLY}.
\end{proof}

\begin{theorem} \label{mixeddegree5}
Every polynomial identity of degree $\le 5$ relating the bilinear operation $[x,y] = \alpha( x \wedge y)$
and the trilinear operation $(x,y,z) = \beta( x \wedge y ) \cdot z$ on the $\s$-module $V(6)$ is a consequence 
of the following identities:
  \begin{enumerate}
  \item
  the anticommutativity of $[x,y]$, 
  \item
  the skew-symmetry of $(x,y,z)$, 
  \item
  the mixed Jacobi identity in degree 3 (Definition \ref{definitionLY}),
  \item
  the Malcev identity in degree 4 (Definition \ref{definitionMalcev}), 
  \item
  the two mixed identities in degree 4 (Definition \ref{definitionLY}), 
  \item
  the 18-term identity in degre 5 (Definition \ref{18termidentity}), 
  \item
  the ternary derivation identity in degree 5 (Definition \ref{definitionLY}),
  \end{enumerate}
and the new 31-term mixed identity in degree 5 displayed in Table \ref{newmixeddegree5}.
\end{theorem}

\begin{proof}
In degree 5, there are 13 association types involving both operations,
  \begin{alignat*}{4}
  &[ [ [ [ -, - ], - ], - ], - ], &\quad
  &[ [ ( -, -, - ), - ], - ], &\quad 
  &[ [ [ -, - ], [ -, - ] ], - ], &\quad 
  &[ ( [ -, - ], -, - ), - ], 
  \\
  &[ ( -, -, [ -, - ] ), - ], &\quad 
  &[ [ [ -, - ], - ], [ -, - ] ], &\quad 
  &[ ( -, -, - ), [ -, - ] ], &\quad 
  &( [ [ -, - ], - ], -, - ), 
  \\
  &( ( -, -, - ), -, - ), &\quad 
  &( [ -, - ], [ -, - ], - ), &\quad 
  &( [ -, - ], -, [ -, - ] ), &\quad 
  &( -, -, [ [ -, - ], - ] ), 
  \\
  &[ -, -, [ -, -, - ] ], 
  \end{alignat*}
with $60 + 60 + 15 + 60 + 30 + 30 + 30 + 60 + 60 + 15 + 30 + 30 + 30 = 510$ multilinear monomials.  
We create a matrix of size $580 \times 510$ 
with a $510 \times 510$ upper block and a $70 \times 510$ lower block; 
this allows us to perform 10 iterations of fill-and-reduce simultaneously.  
Fill-and-reduce using the Hermite normal form achieves rank 123 after the 
second group of 10 iterations; the rank does not increase for another 10 groups of 10.
Let $A$ denote the resulting $123 \times 510$ integer matrix; the nullspace has dimension $510 - 123 = 387$.
The Maple command,
  \[
  \texttt{U := HermiteForm(Transpose(A),output='U',method='integer[reduced]'):}
  \]
computes a $510 \times 510$ integer matrix $U$ for which $U A^t = H$, where $H$ is the Hermite normal form
of $A^t$.  The last 387 rows of $U$ form a lattice basis for the integer nullspace of $A$.  We sort 
these rows by increasing Euclidean norm, and store them in a $387 \times 510$ integer matrix $N$.

We next determine the consequences in degree 5 of the known polynomial identities of degrees $\le 4$.
We must consider the following cases:
  \begin{itemize}
  \item
In degree 3 we have the mixed Jacobi identity; 
we use the trilinear operation to lift it directly to degree 5 (obtaining 5 identities),
and we use the bilinear operation to lift it first to degree 4 (obtaining 4 identities) 
and then to degree 5 (obtaining 20 identities).
  \item
In degree 4 we have the Malcev identity and the two mixed identities;
we use the bilinear operation to lift them to degree 5 (obtaining 15 identities).
  \end{itemize}
We have 40 multilinear identities in degree 5 which generate the $S_5$-module
of all consequences of known identities from lower degrees.  
We construct a matrix of size $630 \times 510$ with a $510 \times 510$ upper block and a $120 \times 510$
lower block.  For each of the 40 lifted identities, we fill the rows of the lower block with the 
coefficient vectors of all permutations of the identity and compute the row canonical form using
modular arithmetic.  When we are done, the rank is 341; this is the dimension of the 
$S_5$-module of consequences of known identities from lower degrees.  

The third stage is to process the known identities of degree 5: 
the 18-term identity and the ternary derivation identity.
The rank increases to 367:
there is a complementary subspace of dimension $387 - 367 = 20$ of new identities in degree 5.

The fourth stage is to apply the procedure of the second stage to the identities corresponding to the rows 
of the matrix $N$ obtained from the first stage.
Starting from the matrix of rank 367, we find that exactly one row of the $N$ increases in the rank:
row 269 increases the rank to 387, the dimension of the nullspace.
This row represents the multilinear polynomial identity of Table \ref{newmixeddegree5}.
\end{proof}

  \begin{table}
  \begin{align*}
  &
  - [ \, [ \, [ \, [ \, a, b \, ], e \, ], d \, ], c \, ] 
  - [ \, [ \, [ \, [ \, a, c \, ], b \, ], d \, ], e \, ] 
  - [ \, [ \, [ \, [ \, a, c \, ], e \, ], d \, ], b \, ] 
  + [ \, [ \, [ \, [ \, a, d \, ], c \, ], b \, ], e \, ]
  \\
  & 
  + [ \, [ \, [ \, [ \, a, e \, ], d \, ], b \, ], c \, ] 
  + [ \, [ \, [ \, [ \, a, e \, ], d \, ], c \, ], b \, ] 
  + [ \, [ \, [ \, [ \, b, c \, ], d \, ], a \, ], e \, ] 
  - [ \, [ \, [ \, [ \, b, d \, ], a \, ], c \, ], e \, ] 
  \\
  & 
  - [ \, [ \, [ \, [ \, b, d \, ], a \, ], e \, ], c \, ] 
  - [ \, [ \, [ \, [ \, c, d \, ], a \, ], e \, ], b \, ] 
  - [ \, [ \, [ \, [ \, d, e \, ], b \, ], a \, ], c \, ] 
  - [ \, [ \, [ \, [ \, d, e \, ], c \, ], a \, ], b \, ] 
  \\
  & 
  - [ \, [ \, ( \, a, c, d \, ), b \, ], e \, ] 
  - [ \, [ \, ( \, a, c, d \, ), e \, ], b \, ] 
  - [ \, [ \, ( \, a, d, c \, ), e \, ], b \, ] 
  + [ \, [ \, ( \, b, c, d \, ), e \, ], a \, ] 
  \\
  & 
  - [ \, [ \, ( \, b, d, c \, ), a \, ], e \, ] 
  - [ \, [ \, ( \, c, e, d \, ), a \, ], b \, ] 
  - [ \, [ \, ( \, d, e, c \, ), a \, ], b \, ] 
  - [ \, [ \, ( \, d, e, c \, ), b \, ], a \, ] 
  \\
  & 
  + [ \, [ \, [ \, a, b \, ], [ \, c, d \, ] \, ], e \, ] 
  + [ \, [ \, [ \, a, d \, ], [ \, b, e \, ] \, ], c \, ] 
  + [ \, [ \, [ \, a, d \, ], [ \, c, e \, ] \, ], b \, ] 
  \\
  & 
  + [ \, ( \, a, d, c \, ), [ \, b, e \, ] \, ] 
  - [ \, ( \, b, c, d \, ), [ \, a, e \, ] \, ] 
  - [ \, ( \, b, d, c \, ), [ \, a, e \, ] \, ] 
  - [ \, ( \, c, e, d \, ), [ \, a, b \, ] \, ] 
  \\
  & 
  + ( \, [ \, [ \, a, b \, ], e \, ], c, d \, )
  + ( \, [ \, [ \, a, e \, ], b \, ], c, d \, ) 
  + ( \, [ \, [ \, a, e \, ], b \, ], d, c \, ) 
  + ( \, [ \, [ \, b, e \, ], a \, ], d, c \, )
  \end{align*}
  \caption{The new identity in degree 5 relating the two operations}
  \label{newmixeddegree5}
  \end{table}

\begin{theorem}
Every polynomial identity of degree $\le 6$ relating the bilinear operation $[x,y] = \alpha( x \wedge y)$
and the trilinear operation $(x,y,z) = \beta( x \wedge y ) \cdot z$ on the $\s$-module $V(6)$ is a consequence 
of the identities listed in Theorem \ref{mixeddegree5} together with a 14-dimensional space of new identities
in degree 6.
\end{theorem}

\begin{proof}
There are 38 association types and 7245 multilinear monomials.
Fill-and-reduce stabilizes at rank 751, indicating a nullspace of dimension 6494.
In this case we have a large number of liftings of identities from lower degrees:
  \begin{enumerate}
  \item
The mixed Jacobi identity in degree 3 can be lifted to degree 6 using either three binary liftings,
a binary lifting followed by a ternary lifting, or a ternary lifting followed by a binary lifting.
  \item
The Malcev identity and the two Lie-Yamaguti identities of degree 4 can be lifted to degree 6 using
either two binary liftings, or one ternary lifting.
  \item  
The 18-term identity, the ternary derivation identity, and the 31-term identity of degree 5 can be
lifted to degree 6 using one binary lifting.  
  \end{enumerate}
The resulting 300 liftings generate a subspace of dimension 6480, indicating a complementary
14-dimensional subspace of new identities.
The computational details are similar to those of previous proofs, and are therefore omitted.
\end{proof}

\begin{problem}
Find a minimal set of $S_6$-module generators for the space of multilinear polynomial identities
of degree 6 relating $[x,y]$ and $(x,y,z)$. 
\end{problem}

\begin{problem}
Use the representation theory of the symmetric group as in Bremner and Hentzel \cite{BremnerHentzel}
to attempt to extend these computations to degree 7.
\end{problem}


\section*{Acknowledgements}

Murray Bremner was partially supported by NSERC.
Andrew Douglas was partially supported by PSC/CUNY.
We thank Ivan Shestakov for telling us about the Filippov $h$-polynomial and Conjecture \ref{filippovconjecture}.


\end{document}